\documentclass[12pt,oneside]{amsart}

\usepackage{cite}

\usepackage[charter]{mathdesign}
\usepackage{tikz}
\usetikzlibrary{positioning, calc}

\title{A primer on chainmails: structures for point-free connectivity}

\author{Jean F. Du Plessis}
\address{Center for Theoretical Physics, Massachusetts Institute of Technology, Cambridge, MA 02139, USA}
\email{jeandp@mit.edu}

\author{Zurab Janelidze}
\address{Department of Mathematical Sciences\\ Stellenbosch University, South Africa 
and  
National Institute for Theoretical and Computational Sciences (NITheCS), Stellenbosch, South Africa}
\email{zurab@sun.ac.za}

\author{Bernandus A. Wessels}
\address{Department of Mathematical Sciences\\ Stellenbosch University, South Africa}
\email{a3wes@icloud.com}

\makeatletter
\@namedef{subjclassname@2020}{
  \textup{2020} Mathematics Subject Classification}
\makeatother

\subjclass[2020]{54D05, 06A06, 06B23, 18A40, 06D22}
\keywords{}

\usepackage{amssymb, thmtools, color, enumitem, hyperref, graphicx}
\usepackage{tasks}
\usepackage{float}
\usepackage{amsfonts}
\usepackage[all,cmtip]{xy}
\hypersetup{pdfborder=0 0 0}
\usepackage{booktabs} 

\usepackage{hyperref}

\usepackage{geometry}
\setlist{leftmargin=1cm}

\declaretheoremstyle[
  headformat = \textcolor{red}{\NUMBER. }\NAME{}\NOTE,
  headindent=0.8cm
]{actstyle}
\theoremstyle{actstyle}
\newtheorem{theorem}{Theorem}
\newtheorem{lemma}[theorem]{Lemma}
\newtheorem{corollary}[theorem]{Corollary}

\newtheorem{remark}[theorem]{Remark}

\newtheorem{example}[theorem]{Example}

\newtheorem{definition}[theorem]{Definition}

\newcommand{\powerset}[1]{\operatorname{\mathcal{P}}(#1)}

\begin{document}
\vspace{-\baselineskip}
\vspace{-\baselineskip}
\vspace{-\baselineskip}

\null\hfill\begin{tabular}[t]{l@{}}
  \text{MIT-CTP/5728}
\end{tabular}

\vspace{\baselineskip}
\vspace{\baselineskip}

\maketitle

\vspace{-\baselineskip}

\begin{abstract} In point-free topology, one abstracts the poset of open subsets of a topological space, by replacing it with a frame (a complete lattice, where meet distributes over arbitrary join). In this paper we propose a similar abstraction of the posets of connected subsets in various space-like structures. The analogue of a frame is called a \emph{chainmail}, which is defined as a poset admitting joins of its \emph{mails}, i.e., subsets having a lower bound. The main result of the paper is an equivalence between a subcategory of the category of complete join-semilattices and the category of chainmails.
\end{abstract}

\makeatletter
  \def\l@subsection{\@tocline{2}{0pt}{4pc}{5pc}{}}
\makeatother

\section*{Introduction}

\textbf{This paper was a work in progress and has been superseded by \href{https://arxiv.org/abs/2501.19226}{arXiv:2501.19226}}

Point-free topology generalizes various constructions, properties, and results about topological spaces to frames (recall that a frame is a complete lattice where finite meets distribute over infinite joins). The generalization is given by replacing the frame of open subsets of a topological space, ordered under subset inclusion, with an abstract frame. See \cite{picado_pultr_2012} for an overview of this theory.

With this paper, we propose a similar abstract study of connectedness in a topological space, or any other space-like structure, where the poset of connected subsets is to be replaced with an abstract poset. Abstract approaches to connectedness have been explored previously (see \cite{10.2307/1969257,article,Hammer1968Sep,Serra1998Nov} or see \cite{stadler_stadler_2015} for a broader summary). Our work can be considered as a next step after \cite{Serra1998Nov} (a reference, which we became aware of only when finalizing this paper) that defines connectivity abstractly, but in the context of an ambient complete lattice. The notion of abstract connectivity that we formulate in this paper is defined by similar axioms as those in \cite{Serra1998Nov}, except that we discard the ambient lattice; we show, however, that an ambient lattice can always be canonically produced. 

In our approach to abstract connectivity we introduce \emph{chainmails}, which are posets having the following property:
\begin{itemize}
   \item[(C)] Every \emph{mail-connected} subset of the poset has a join.
\end{itemize}
Here `mail-connected' refers to the graph-theoretic connectedness of non-empty subsets of the poset, where the graph is given by elements of the poset as vertices, with a pair $\{v_1,v_2\}$ of vertices being an edge when $v_1^\downarrow \cap v_2^\downarrow\neq\varnothing$ (see Figure~\ref{fig:mail-connectedness}). Note that if in (C) we drop `mail-connected', we get the definition of a complete lattice. 

In various contexts where connectivity can be defined, the posets of connected objects always form a chainmail. This includes all previous point-set approaches to abstract connectivity (see \cite{stadler_stadler_2015}), the point-free approach taken by Serra in \cite{Serra1998Nov}, and all classical situations which are examples of the former: usual connectivity in topological spaces, path-connectivity in spaces, connectivity in graphs and others. Among examples covered by Serra are also various forms of connectivity arising in digital imaging (see \cite{Serra1998Nov}).

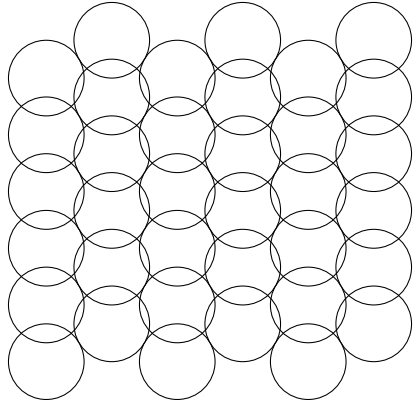
\begin{figure}
\[
\begin{tikzpicture}
    \def\n{5} 
    \def\radius{0.5} 
    \foreach \row in {0,...,\n}
        \foreach \col in {0,...,\n} {
            \pgfmathsetmacro\xshift{\col*2*\radius*cos(30)}
            \pgfmathsetmacro\yshift{\row*1.5*\radius}
            \ifodd\col
                \pgfmathsetmacro\yshift{\yshift+\radius}
            \fi
            \draw (\xshift, \yshift) circle (\radius);
        }
\end{tikzpicture}
\]
\caption{Venn diagram visualization for the notion of mail-connectedness}
\label{fig:mail-connectedness}
\end{figure}

In this paper we show that there is an adjunction between the category of chainmails, and a category of join-complete lattices with appropriate morphisms. The adjunction is given, in one direction, by assigning to a complete lattice the poset of its (suitably defined) connected elements, which turns out to be a chainmail, and in the other direction, by assigning to a chainmail the poset of its down-closed subchainmails, which happens to be a complete lattice. A standard fact in category theory is that every adjunction restricts to an equivalence of categories. In our case, we obtain an equivalence between the category of all chainmails and a suitable subcategory of the category of join-complete lattices, given by `locally connected' complete lattices (a generalisation of locally connected frames, see e.g.~\cite{picado_pultr_2012}). 

\begin{table}[ht]
    \centering
    \begin{tabular}{|r||r|r|r|r|r|r|r|r|r|r|}
        \hline
        $n$ & 1 & 2 & 3 & 4 & 5 & 6 & 7 & 8 & 9 & 10\\\hline
        Mail-connected chainmails & 1 & 1 & 2 & 5 & 16 & 62 & 303 & 1842 & 14073 & 134802\\\hline
    \end{tabular}
    \caption{Number of mail-connected (single component) chainmails of size $n$}
    \label{tab:chml_counts}
\end{table}

\section{Connectivity in complete lattices}

Given a complete lattice $L$, by a \emph{separated} set we mean a subset $S\subseteq L$ such that $0\notin S$ and for any two distinct elements $s,s'\in S$, we have $s\wedge s'=0$. Consider the following conditions on an element $a$ in $L$.

\begin{itemize}
    \item[(E1)] $a\neq 0$ and for any $x,y$, if $a\le x\vee y$ and $x\wedge y=0$, then $a\le x$ or $a\le y$.

    \item[(E2)] $a\neq 0$ and for any $x,y$, we have:
    \[\big[x\wedge y=0\text{ and }a=x\vee y\big]\implies \big[x=0\text{ or } y=0\big].\]
    \item[(E3)] for any separated set $S$, if $a=\bigvee S$ then $a\in S$ (and consequently, $S=\{a\}$).
    \item[(E4)] for any separated set $S$, if $a\le \bigvee S$ then $a\le s$ for some $s\in S$.
\end{itemize}

It is easy to see that in a general complete  lattice $L$ we still have the following implications:
\[\textrm{(E1) implies (E2), 
(E3) implies (E2), and 
(E4) implies (E1-3).}\] We say that an element $a$ in $L$ is \emph{connected} when (E4) holds, or equivalently, all of (E1-4) hold.

\begin{definition}
A \emph{locally connected lattice} is a complete lattice where every element is a join of connected elements.     
\end{definition}

A \emph{chained set} in a complete lattice $L$ is a non-empty subset $C$ of $L$ such that for any two elements $x,y\in C$ there exists a sequence $x=c_0,c_1,\dots,c_n=y$ such that $c_i\in C$ and $c_{i-1}\wedge c_i\neq 0$ for each $i\in\{1,\dots,n\}$. Note that singletons whose element is not $0$ are both chained and separated, and moreover, they are the only chained separated sets.

The following result is Lemma~1.2 in \cite{BABOOLAL19913}, but we prove a slightly stronger form here.

\begin{lemma}\label{lemC}
In a complete lattice, the join of a chained set of connected elements is connected.
\end{lemma}

\begin{proof}
Consider the join $\bigvee C$ of a chained set $C$ of connected elements. Let $S$ be a separated set such that $\bigvee C\leqslant \bigvee S$. Then $x\leqslant \bigvee C$ for each $x\in C$. Since each $x\in C$ is connected, we get that $x\leqslant s^x$ for some $s^x\in S$. Consider any two $x,y\in C$. Since $C$ is chained, there is a sequence $x=c_0,c_1,\dots,c_n=y$ in $C$ such that $c_{i-1}\wedge c_i\neq 0$ for each $i\in\{1,\dots,n\}$. Since $S$ is separated, this forces
\[s^x=s^{c_0}=s^{c_1}=\dots=s^{c_n}=s^y.\]
Thus, $s^x$ is the same element $s$ of $S$ for each $x
\in C$. Then, $\bigvee X\leqslant s$, proving that $\bigvee X$ is connected.
\end{proof}

An element $x$ in a locally connected lattice can be presented as the join of connected elements below it. We can partition the set of connected elements below $x$ according to maximal chained sets of connected elements below $x$. Consider the join $c$ for each component $C$. By the lemma above, $c$ will itself be connected. From this it follows easily that $c\in C$. We denote the set of such $c$'s by $x^\ast$. 

\begin{lemma}\label{lemB}
For any element $x$ in a locally connected lattice $L$, the set $x^\ast$ is separated, and moreover, $x=\bigvee x^\ast$.   
\end{lemma}

\begin{proof} Firstly, we note that since each element $c$ of $x^\ast$ is connected, we have: $0\notin x^\ast$.
Suppose $c\wedge d\neq 0$ for some $c,d\in x^\ast$. Then there is a connected $e\leqslant c\wedge d$. This forces $c$ and $d$ to belong to the same maximal chained sets of elements below $x$, and hence, $c=d$. Therefore, $x^\ast$ is separated. Note that every element of $x^\ast$ is below $x$. Since every connected element below $x$ is below one of the elements of $x^*$, and since $x$ is a join of connected elements, it then follows that $x=\bigvee x^\ast$. 
\end{proof}

The lemma below shows that the set $x^*$ is unique.

\begin{lemma}\label{lem:unique_separation}
    For any separated sets $S$ and $T$ of connected elements in a lattice, we have that:
    \[\left[\bigvee S=\bigvee T\right]\implies S=T.\]
\end{lemma}

\begin{proof}
    For any $s_1\in S$ we have that $s_1\le \bigvee S=\bigvee T$. Since, $s_1$ is connected there must be some $t\in T$ such that $s_1\le t$. Applying the same logic to $t$, there must be some $s_2$ such that $s_1\le t\le s_2$. Since $S$ is separated, we must have that $s_1=s_2=t$. Now $s_1\in T$ and the same reasoning applies to any element in $T$, so $S=T$.
\end{proof}

It turns out that like in a frame, in a locally connected lattice (E1-4) are equivalent.

\begin{theorem}
In a locally connected lattice $L$, any element $a$ satisfying (E2) is connected. 
\end{theorem}

\begin{proof}
Consider an element $a$ satisfying (E2). From the previous lemma, $a=\bigvee a^\ast$. Consider any element $c\in a^\ast$ and the set $S=a^\ast\setminus\{c\}$. If $c\wedge \bigvee S\neq 0$ then there is a connected element $c'$ below $c\wedge \bigvee S$. Since $S$ is separated, $c'\leqslant s$ for some $s\in S$ such that $s\neq c$. Then $c'\le c\wedge s=0$, a contradiction. So $c\wedge \bigvee S=0$. But $a=c\vee \bigvee S$, so by (E2), $\bigvee S=0$. This implies that $S$ is empty, and so $c$ is the only element of $a^\ast$, which in turn implies that $a=c$ and $a$ is connected.
\end{proof}

\begin{corollary}
In a locally connected lattice, the conditions (E1-4) are equivalent.
\end{corollary}

\begin{example}
In the complete lattice of closed sets of a Hausdorff space, where singletons are closed sets, a closed set $a$ satisfies (E1) if and only if it satisfies (E2) and if and only if it is a connected closed set; $a$ satisfies (E3) if and only if it is a singleton. It satisfies (E4) if and only if it is a clopen singleton. Thus, this lattice is locally connected if and only if the Hausdorff space is discrete, in which case the lattice is just the power set of the underlying set.
\end{example}

For a complete lattice, by $\mathcal{S}(L)$ we denote the poset of separated sets of connected elements, where $S_1\le S_2$ for two such sets provided every element of $S_1$ is below some element of $S_2$. There is an obvious monotone function $\nu_L\colon \mathcal{S}(L)\to L$ which maps each separated set to its join.

\begin{theorem}\label{thmD}
For a complete lattice $L$ the following conditions are equivalent: 
\begin{enumerate}
    \item $L$ is locally connected.
    \item $\nu_L$ is an isomorphism.
    \item $\nu_L$ is surjective.
\end{enumerate}
\end{theorem}

\begin{proof} The implications
(2)$\Rightarrow$(3) and (3)$\Rightarrow$(1) are obvious. We prove that $\nu_L$ is an isomorphism for any locally connected lattice $L$. For this, it is sufficient to prove that $\nu_L$ is surjective and reflects the order. In fact, surjectivity readily follows from Lemma~\ref{lemB}. Let $S_1$ and $S_2$ be two separated sets such that $\bigvee S_1\le \bigvee S_2$. Then for each $s_1\in S_1$, we have $s_1\le \bigvee S_2$. Since each $s_1\in S_1$ is connected, this implies that for each $s_1\in S_1$ there exists $s_2\in S_2$ such that we have $s_1\le s_2$, in other words, $S_1\le S_2$.       
\end{proof}

\section{Chainmails}\label{secA}  

\begin{figure}
    \centering
    \includegraphics[width=0.98\textwidth]{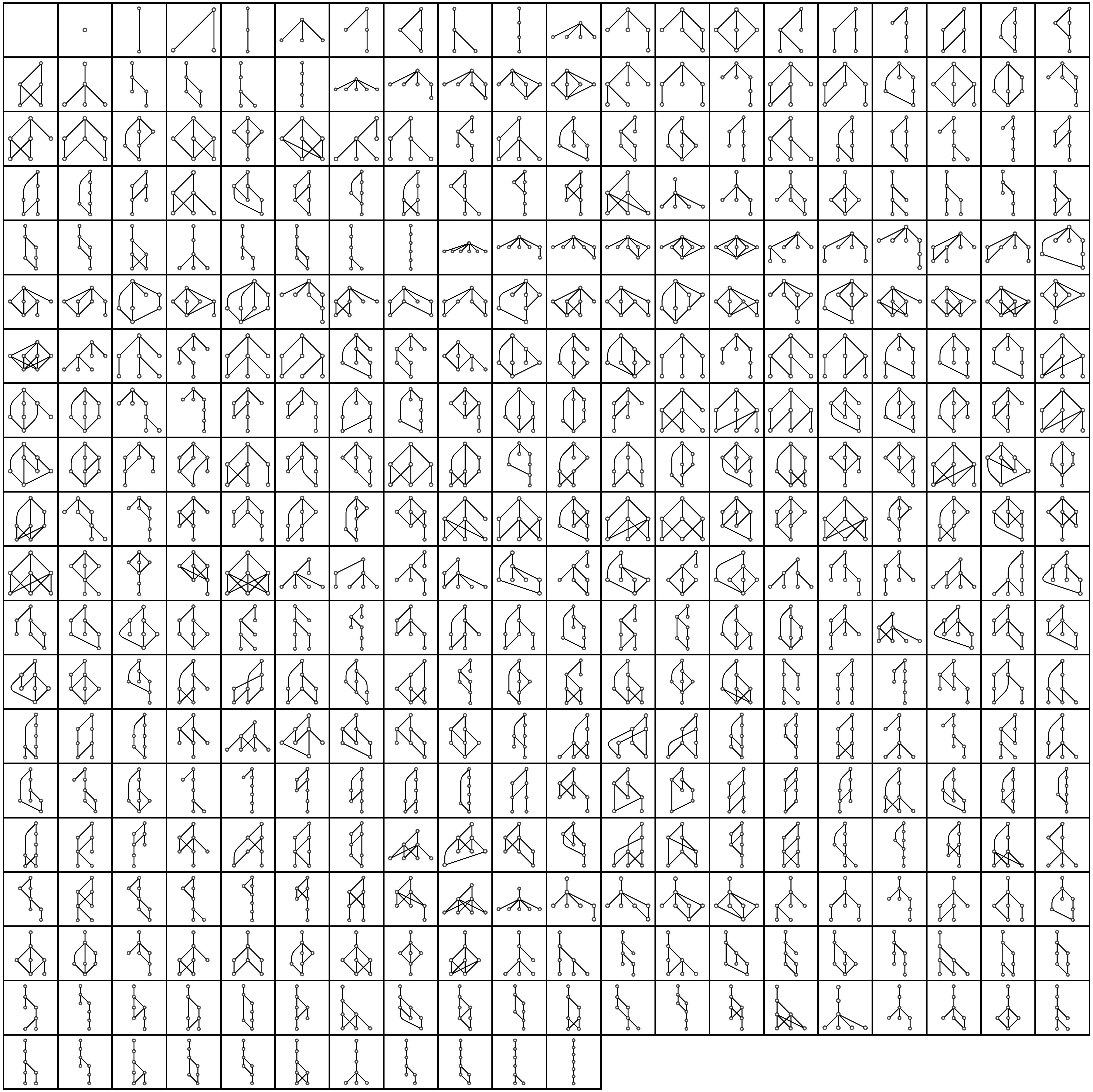}
    \caption{All mail-connected (single component) chainmails with at most 7 elements.}
    \label{fig:chainmail_table}
\end{figure}

\begin{definition}\label{def:Chainmail}
In a poset, a \emph{mail} is a non-empty set of elements of the poset having a common lower bound. A \emph{chainmail} is a poset in which every mail has a join.
\end{definition}

Intuitively, in a space-like mathematical structure, a connected region is one that cannot be `cleanly parted' into two non-empty regions. Suppose two connected regions have a shared connected sub-region. Any way of parting the combination of these two regions has to part at least one of the regions, or their shared connected region. Therefore, the region obtained by combining the two regions is also connected. This intuition, which is embodied in Definition~\ref{def:Chainmail}, covers essentially all standard notions of connectedness as we will see in the examples below. 

\begin{example}\label{exaB}
A subset $C$ of a graph is connected when it is not empty and for any two elements in the subset, there is a path of edges connecting them, where all the vertices of the path are in $C$. Connected sets of a graph form a poset under subset inclusion. Now, a mail in this poset is a set of connected sets whose intersection contains a connected set, and hence at least one vertex $v$. The union of this set is then connected, since any two elements in the union can be connected with a path passing through $v$. This implies that the poset of connected subsets in a graph is a chainmail.
\end{example}

\begin{example}
Connectivity in a graph generalises to connectivity in a `hypergraph', i.e., a set $H$ equipped with a distinguished set $E$ of subsets called `hyperedges'. Define a non-empty subset $C$ of $H$ to be connected, where for any two elements $x,y\in C$ there is a sequence of hyperedges $E_0,\dots,E_n\in E$ such that $x\in E_0$, $y\in E_n$, any two consecutive hyperedges in the sequence have nonempty intersection, and each hyperedge is a subset of $C$. The poset of connected sets is then a chainmail. To get connectivity in a graph, simply consider the hypergraph where the hyperedges are the singletons and two-element sets of vertices connected by an edge.    
\end{example}

\begin{example}\label{exaC}
The poset of connected subsets of a topological space is a chainmail. This follows from the fact that the union of a mail of connected sets is connected. Indeed, suppose the union is a subset of the union of two disjoint open sets $A$ and $B$. Since there is a connected set which is a subset of every element of the mail, there is a point $x$ which is contained in every element of the mail. This point belongs either to $A$ or $B$. If it belongs to $A$, then every element of the mail will be forced to belong to $A$, by their connectivity. So the union of the elements of the mail belongs to $A$. Similarly, for $B$.
\end{example}

\begin{example}\label{exaD}
The union of a mail of path-connected subsets of a topological space will be path-connected, since any two points in the union can be joined by a path that passes through an element of the lower bound of elements of the mail. Thus, path-connected sets in a topological space form a chainmail.
\end{example}

A \emph{connectivity space} in the sense of \cite{stadler_stadler_2015} is a pair $(X,\mathcal{C})$ where $X$ is a set and $\mathcal{C}\subseteq\powerset{X}$ is a set of subsets of $X$ such that:
\begin{enumerate}
\item[(c0)] $\varnothing\in\mathcal{C}$,
\item[(c1)] $Z\subseteq\mathcal{C}$ and $\bigcap Z\neq\varnothing$ implies $\bigcup Z\in\mathcal{C}$.
\end{enumerate}
The poset of non-empty elements of $\mathcal{C}$, which we call \emph{connected sets} of the connectivity space, is a chainmail: axiom (c1) implies that the union of a mail of connected sets is connected, and hence it is the join of the mail in the poset of connected sets. All previous examples of chainmails arise in this way from suitable connectivity spaces: in each case, take $X$ to be the underling set of the space-like structure and let $\mathcal{C}$ be the set of connected subsets along with the empty set. One may be tempted to think that every chainmail can be obtained in this way from a connectivity space. With the next example, we show that this is not the case.

\begin{example}\label{exaA}
Consider the following poset:
\[\vcenter{\hbox{\begin{tikzpicture}[
    dot/.style={circle, draw=black, fill=black, minimum size=0em, inner sep=0.1em}]
    
    \node[dot,label=right:1] (1) {};
    \node[dot,label=right:3] (3) [above=of 1]{};
    \node[dot,label=right:2] (2) [left=of 3]{};
    \node[dot,label=right:4] (4) [right=of 3]{};
    \node[dot,label=right:5] (5) [above=of 3]{};
    \node[dot,label=right:6] (6) [above=of 4]{};
    \node[dot,label=right:7] (7) [above=of 5]{};
    
    \draw[-] (1)--(2)--(5)--(7);
    \draw[-] (1)--(3)--(5);
    \draw[-] (3)--(6)--(7);
    \draw[-] (4)--(5);
    \draw[-] (4)--(6);
    \end{tikzpicture}}}\]
Since taking away an element from a set of elements in a poset which is below another element in the same set does not effect the join of the set, to check that the poset above is a chainmail it is sufficient to check that joins exist for all \emph{reduced mails}, i.e., mails where no two elements are comparable. Apart from the singletons and mails containing the top element, in which case joins trivially exist, such mails along with their joins are: \[\bigvee \{2,3\}=5, \bigvee \{2,6\}=7, \bigvee \{5,6\}=7.\]
Thus, the poset above is indeed a chainmail. This chainmail cannot be a chainmail of connected sets in any connectivity space. Indeed, the connected set $4$, being a subset of the union of $2$ and $3$, must contain an element from $3$ (since otherwise it would have been a subset of $2$). But then, $3$ and $4$ must have a join in the chainmail, which they do not.
\end{example}

The following graph structure on a poset will be useful: an edge between two distinct elements $c,d$ represents the fact that the set $\{c,d\}$ is a mail (or equivalently, is a subset of a mail). Subsets of the poset which are connected for this graph structure will be called \emph{mail-connected sets} of the poset (as in the Introduction).
Thus, a set $C$ of elements of a poset is mail-connected if it is not empty and for any two elements $c,d$ in the set, there is a sequence $c=x_0,x_1,\dots,x_n=d$ in the set such that $x_{i-1},x_i$ have a common lower bound (not necessarily in $C$), for each $i\in\{1,\dots,n\}$. Such sequence is called a 
\emph{path} of length $n$, corresponding to a usual path in the graph described above. 

\begin{theorem}\label{thmB} In a chainmail, every mail-connected set has a join.
\end{theorem}

\begin{proof} In a chainmail, every path has a join. It can be obtained by a hierarchically joining mails, as illustrated in the case of a path of length four:
\[
\begin{tikzpicture}[node distance=0.7cm]

\node (a1) {$a_0$};
\node[right=of a1] (a2) {$a_1$};
\node[right=of a2] (a3) {$a_2$};
\node[right=of a3] (a4) {$a_3$};
\node[right=of a4] (a5) {$a_4$};

\node[above=of $(a1)!0.5!(a2)$] (b1) {$b_0$};
\node[above=of $(a2)!0.5!(a3)$] (b2) {$b_1$};
\node[above=of $(a3)!0.5!(a4)$] (b3) {$b_2$};
\node[above=of $(a4)!0.5!(a5)$] (b4) {$b_3$};

\draw (a1) -- (b1) -- (a2);
\draw (a2) -- (b2) -- (a3);
\draw (a3) -- (b3) -- (a4);
\draw (a4) -- (b4) -- (a5);

\node[above=of $(b1)!0.5!(b2)$] (c1) {$c_0$};
\node[above=of $(b2)!0.5!(b3)$] (c2) {$c_1$};
\node[above=of $(b3)!0.5!(b4)$] (c3) {$c_2$};

\draw (b1) -- (c1) -- (b2);
\draw (b2) -- (c2) -- (b3);
\draw (b3) -- (c3) -- (b4);

\node[above=of $(c1)!0.5!(c2)$] (d1) {$d_0$};
\node[above=of $(c2)!0.5!(c3)$] (d2) {$d_1$};

\draw (c1) -- (d1) -- (c2);
\draw (c2) -- (d2) -- (c3);

\node[above=of $(d1)!0.5!(d2)$] (top) {$e$};

\draw (d1) -- (top);
\draw (d2) -- (top);

\end{tikzpicture}
\]
Consider a mail-connected set $C$ and pick any element $c\in C$. For each $d\in C$ there is a path that connects $c$ with $d$. Consider the set of all joins of such chains. It is a mail whose join will be the join of $C$.
\end{proof}

\begin{corollary}
A poset is a chainmail if and only if it satisfies (C).
\end{corollary}

Every mail-connected set in the poset $\mathcal{K}(L)$ of connected elements of a complete lattice $L$ is a chained set in $L$, and so, it follows from Lemma~\ref{lemC} that $\mathcal{K}(L)$ is always a  chainmail, where joins of mail-connected sets are given by joins in $L$. When the lattice $L$ is locally connected, mail-connected sets in $\mathcal{K}(L)$ are the same as chained sets of connected elements in $L$; in fact, for this it is sufficient to require that every element in $L$ that is different from $0$ has a connected element below it --- we say in this case that $L$ has a \emph{connective foundation}. 

A set $S$ of elements of a poset is said to be \emph{totally disconnected} if no two distinct elements of $S$ have a common lower bound, or equivalently, its mail-connected subsets are only the singletons. Every singleton is both mail-connected and totally disconnected, and moreover, singletons are the only such sets.
. For a chainmail $\Gamma$, we write $\mathcal{D}(\Gamma)$ to denote the poset of totally disconnected sets in $\Gamma$: $D_1\le D_2$ for two totally disconnected sets $D_1$ and $D_2$, whenever each element of $D_1$ is below some element of $D_2$.

For a complete lattice $L$, every separated set of connected elements in $L$ will be a totally disconnected set in $\mathcal{K}(L)$. When $L$ has a connective foundation, the converse is also true: every totally disconnected set in $\mathcal{K}(L)$ is a separated set of connected elements in $L$. So in this case, $\mathcal{S}(L)=\mathcal{D}\mathcal{K}(L)$.  

Using the theorem below, it turns out that $\mathcal{D}(\Gamma)$ is a complete lattice for any chainmail $\Gamma$.

\begin{theorem}\label{thmF} For a down-closed set $X$ in a chainmail $\Gamma$, the following conditions are equivalent:
\begin{enumerate}
\item $X$ is a down-set of a totally disconnected set in $\Gamma$.

\item The join of every subset of $X$ that is a mail belongs to $X$.

\item $X$ is closed under joins of mail-connected subsets of $X$.
\end{enumerate}
Furthermore, the poset $\mathcal{D}(\Gamma)$ is isomorphic to the poset of those down-closed sets $X$ in $\Gamma$ which satisfy the equivalent conditions above. The isomorphism is given by mapping a totally disconnected set to its down-set, and backwards, by mapping a down-closed set $X$ satisfying the equivalent conditions above to the set $X^\ast$ of joins of maximal mail-connected subsets of $X$.
\end{theorem}

\begin{proof}
(1)$\Rightarrow$(2): Suppose $X$ is a down-set of totally disconnected set $D$ in $\Gamma$. Consider a subset $M$ of $X$ that is a mail, and let $q$ be a lower bound of $M$. Each element of $M$ must be below one of the elements in $D$. This would force $q$ to be below all such elements of $D$, and so those must be the same element since $D$ is totally disconnected. The join of $M$ will then also be below that element, which means that it is contained in $X$.

(2)$\Rightarrow$(3): This can be proved by adapting the argument from the proof of Theorem~\ref{thmB}.

(3)$\Rightarrow$(1): Suppose $X$ is as in (3). Consider the maximal mail-connected subsets of $X$. The join of each such subset is then the top element in it. Furthermore, the set $D$ of these joins must be totally disconnected. $X$ is then the down-set of $D$.

The last two statements of the theorem follow easily from the reduction just described of an $X$ satisfying (3) into a separated set.
\end{proof}

Notice that down-closed subsets $X$ of a chainmail satisfying the equivalent conditions will be chainmails in their own right. We will henceforth refer to them simply as \emph{subchainmails}. Thus, by the theorem above, $\mathcal{D}(\Gamma)$ is isomorphic to the poset of subchainmails of $\Gamma$. In the future, having in mind this isomorphism, we will often work in the latter poset when we want to establish properties of $\mathcal{D}(\Gamma)$. The first illustration of this is the following. Condition (2) in the theorem above makes it particularly evident that arbitrary intersection of subchainmails is a subchainmail. We then obtain:

\begin{corollary}
For any chainmail $\Gamma$, the poset $\mathcal{D}(\Gamma)$ is a complete lattice.
\end{corollary}

The join of subchainmails is obtained by generating a subchainmail from the union of the given subchainmails; i.e., by intersecting all subchainmails containing the given subchainmails.

The set $X^\ast$ of joins of maximal mail-connected subsets of a set $X$ can sometimes be totally disconnected and hence, its down-set be the subchainmail generated by $X$. We note the following two cases in which this occurs. Firstly, this is trivially so when $X$ is already a subchainmail. Then $X^\ast$ is the corresponding totally disconnected set whose down-set is $X$ (Theorem~\ref{thmF}). Secondly, it is so when $X$ is totally disconnected, since in that case $X^\ast=X$, and it is easy to see that the down-set of $X$ is then the subchainmail generated by $X$. This includes the case when $X$ consists of a single element and when it is empty: the subchainmail generated by a single element is its down-set, while the subchainmail generated by the empty set is the empty set.

\begin{lemma}\label{lemA}
For any chainmail $\Gamma$, a set $S$ in $\mathcal{D}(\Gamma)$ is separated if and only if it does not contain the empty set, its elements are pair-wise disjoint, and its union $\bigcup S$ is a totally disconnected set in $\Gamma$. When $S$ is separated in $\mathcal{D}(\Gamma)$, its join is given by $\bigcup S$.
\end{lemma}

Below is a translation of the lemma above to the language of subchainmails.

\begin{lemma}\label{lemE}
A set $R$ of subchainmails of a chainmail $\Gamma$ is separated in the complete lattice of subchainmails if and only if it does not contain the empty subchainmail, its elements are pair-wise disjoint, and its union $\bigcup R$ is a subchainmail. When $R$ is separated, $\bigcup R$ is the join of $R$ in the complete lattice of subchainmails.
\end{lemma}

The lemmas above lead us to the following result:

\begin{theorem}\label{thmH} For any chainmail $\Gamma$, connected elements in the complete lattice $\mathcal{D}(\Gamma)$ are singletons, and consequently, the complete lattice $\mathcal{D}(\Gamma)$ is locally connected and moreover, $\Gamma$ is isomorphic to $\mathcal{K}\mathcal{D}(\Gamma)$. This isomorphism maps each element $x$ of $\Gamma$ to the singleton $\{x\}$.
\end{theorem}

Combine this theorem with Theorem~\ref{thmD} and the observation that $\mathcal{S}(L)=\mathcal{D}\mathcal{K}(L)$ as soon as $L$ has a connective foundation (and in particular, when $L$ is a locally connected lattice). We obtain:

\begin{corollary}\label{corA} There is a one-to-one correspondence between isomorphism classes of locally connected lattices and isomorphism classes of chainmails. In the forward direction, it is given by mapping a locally connected lattice $L$ to the chainmail $\mathcal{K}(L)$ of its connected elements, while in the backward direction it is given by mapping a chainmail $\Gamma$ to the poset $\mathcal{D}(\Gamma)$ of its totally disconnected sets.
\end{corollary}

In \cite{upcoming}, we will show that we can characterise the relevant chainmails when we limit this correspondence to frames instead of general complete lattices.

\section{The category of chainmails}

In this section we show that the one-to-one correspondence between chainmails and locally connected lattices obtained in the previous section is in fact an equivalence of categories, and moreover, this equivalence arises as the largest equivalence for an adjunction between suitable categories of complete lattices and chainmails. In fact, we will see that the $\mathcal{D}$ construction allows us to view the category of chainmails as a full mono co-reflective subcategory of the category of complete lattices. Locally connected lattices are just objects in the bigger category that are isomorphic to those in the subcategory.   

Note that any monotone map between posets preserves both mails and mail-connected sets. The definition of a chainmail suggests an obvious notion of a morphism between chainmails:

\begin{definition}
A \emph{chainmail morphism} $\Gamma_1\to \Gamma_2$ is a map $\Gamma_1\to \Gamma_2$ between chainmails $\Gamma_1$ and $\Gamma_2$, which preserves joins of mails.
\end{definition}

Note that chainmails morphisms are monotone.

\begin{example}
A continuous function between topological spaces preserves connectivity and so it gives rise to a chainmail morphism between the corresponding chainmails of connected sets, which maps a connected set in the domain to its direct image in the codomain. The same is true for path-connectivity.    
\end{example}

An argument similar to the one used in the proof of Theorem~\ref{thmB} allows one to prove the following:

\begin{theorem}
Every chainmail morphism preserves joins of mail-connected sets.
\end{theorem}

It is easy to see that chainmails and chainmail morphisms form a category, where composition of morphisms is given by composition of maps. We denote this category by $\mathbf{Chm}$. 

In what follows we will show that $\mathcal{D}$ defines a full and faithful functor from $\mathbf{Chm}$ to the category $\mathbf{Con}$ of complete lattices and ``connectivity homomorphisms'' between complete lattices, defined as follows.

\begin{definition}
A \emph{connectivity homomorphism} from a complete lattice $L$ to a complete lattice $M$ is a join-preserving map $F\colon L\to M$ whose right adjoint $F^\bullet$ preserves joins of separated sets.
\end{definition}

A consequence of the requirement on $F^\bullet$ is that $F$ preserves connected elements. 

\begin{theorem}\label{thmG} For connectivity homomorphism $F:L\to M$ between complete lattices, $F^\bullet$ preserves the bottom element and the separated sets, while $F$ preserves connected elements.
\end{theorem}

\begin{proof} 
Let $F:L\to M$ be a connectivity homomorphism between complete lattices. The empty set in $M$ is separated and its join is the bottom element of $M$. So $F^\bullet$ preserves the bottom element. Since a right adjoint map always preserves meets, we get that $F^\bullet$ preserves separated sets. Let $c$ be a connected element in $L$. Suppose $F(c)\leqslant \bigvee S$, where $S$ is a separated set in $M$. Then \[c\le F^\bullet F(c)\le F^\bullet \bigvee S=\bigvee F^\bullet S.\]
By connectivity of $c$, it is below an element in $F^\bullet S$. But then, $F(c)$ is below an element in $S$, proving that $F(c)$ is connected.
\end{proof}

It is easy to see that complete lattices and connectivity homomorphisms between them form a category, under functional composition of connectivity homomorphisms. We denote this category by $\mathbf{Con}$.

\begin{theorem}\label{thmI}
Assigning to a chainmail morphism $m\colon \Gamma_1\to \Gamma_2$ the function $\mathcal{D}(m)\colon \mathcal{D}(\Gamma_1)\to \mathcal{D}(\Gamma_2)$, which maps a totally disconnected set $D_1$ in $\mathcal{D}(\Gamma_1)$ to the join \[\mathcal{D}(m)(D_1)=\bigvee \{\{m(d)\}\mid d\in D_1\}\] in $\mathcal{D}(\Gamma_2)$, defines a functor $\mathbf{Chm}\to\mathbf{Con}$. The right adjoint $\mathcal{D}(m)^\bullet$ maps a totally disconnected set $D_2$ in $\mathcal{D}(\Gamma_2)$ to
\[\mathcal{D}(m)^\bullet(D_2)=\left(m^{-1}(D_2^\downarrow)\right)^\ast.\] 
This functor has a right adjoint given by assigning to a connectivity homomorphism $F\colon L_1\to L_2$ the function $\mathcal{K}(F)\colon \mathcal{K}(L_1)\to \mathcal{K}(L_2)$ that maps a connected element $c$ in the complete lattice $L_1$ to $F(c)$. The unit $\eta$ of this adjunction is a natural isomorphism; for each chainmail $\Gamma$, it is defined by $\eta_\Gamma(x)=\{x\}$. The counit $\varepsilon$ of this adjunction is a mono, and is defined, for each complete lattice $L$, by $\varepsilon_L(D)=\bigvee D$. 
\end{theorem}

\begin{proof}
In the language of subchainmails, the theorem states that $\mathcal{D}(m)$ will map a subchainmail of $\Gamma_1$ to the subchainmail generated by its image under $m$, while $\mathcal{D}(m)^\bullet$ will map a subchainmail of $\Gamma_2$ to its inverse image under $m$. It is easy to see that these two mappings constitute indeed a Galois connection between $\mathcal{D}(\Gamma_1)$ and $\mathcal{D}(\Gamma_2)$. Since the inverse image under any function preserves unions of disjoint sets, it follows from Lemma~\ref{lemE} that $\mathcal{D}(m)^\bullet$ preserves joins of separated sets. Thus, $\mathcal{D}(m)$ is a connectivity homomorphism. That $\mathcal{D}$ is a functor follows from functoriality of the inverse image map construction.

Every connectivity homomorphism $F\colon L_1\to L_2$ preserves connected elements (Theorem~\ref{thmG}), so $\mathcal{K}(F)$ can be defined as stated in the theorem. A mail $M$ in $\mathcal{K}(L_1)$ is a chained set in $L_1$, so its join is given by the join in $L_1$ (Lemma~\ref{lemC}). Similarly, the join of $F(M)$ in $\mathcal{K}(L_2)$ is given by its join in $L_2$. Since $F$ preserves arbitrary joins, the map $\mathcal{K}(F)$ will then preserve joins of mails. So each $\mathcal{K}(F)$ is indeed a chainmail morphism. It is obvious that $\mathcal{K}$ is a functor.  

Consider $\eta$ as defined in the theorem. That each $\eta_\Gamma$ is an isomorphism is witnessed by Theorem~\ref{thmH}. Naturality of this isomorphism can be established at once:
\[\mathcal{K}(\mathcal{D}(m))\eta_{\Gamma_1}(x)=\mathcal{K}(\mathcal{D}(m))\{x\}=\mathcal{D}(m)\{x\}=\bigvee\{\{m(x)\}\}=\{m(x)\}=\eta_{\Gamma_2}m(x).\]

Consider now $\varepsilon$ as defined in the theorem. Let us prove that each $\varepsilon_{L}$ is a connectivity homomorphism.
First, we construct  $\varepsilon_{L}^\bullet\colon L\to \mathcal{D}\mathcal{K}(L)$. For each element $x\in L$, consider the set $C_x$ of all connected elements below $x$. This set will be a subchainmail of $\mathcal{K}(L)$, thanks to Lemma~\ref{lemC}. Define $\varepsilon_{L}^\bullet(x)=C_x^\ast$ to be the totally disconnected set corresponding to this subchainmail. It is easy to see that $\varepsilon_{L}^\bullet$ is a monotone map. Every element of $\varepsilon_{L}^\bullet(x)$ is a join of a mail-connected subset of $C_x$ in $\mathcal{K}(L)$. Thanks to Lemma~\ref{lemC} again, it is then also a join of a subset of $C_x$ in $L$, and thus below $x$. So $\varepsilon_{L}\varepsilon_{L}^\bullet(x)=\bigvee \varepsilon_{L}^\bullet(x)$ is below $x$. Now, starting with a totally disconnected subset $D$ of $\mathcal{K}(L)$, the set $\varepsilon_{L}^\bullet\varepsilon_{L}(D)=\varepsilon_{L}^\bullet(\bigvee D)$ is the totally disconnected set corresponding to the subchainmail $C_{\bigvee D}$ of $\mathcal{K}(L)$ consisting of connected elements below $\bigvee D$. Since every element of $D$ is connected, $D\subseteq C_{\bigvee D}$. It then follows that $D\le \varepsilon_{L}^\bullet\varepsilon_{L}(D)$. This shows that $\varepsilon_{L}^\bullet$ is a right adjoint of $\varepsilon_{L}$. To show that $\varepsilon_{L}$ is a connectivity homomorphism in $\mathbf{Con}$, it remains to show that $\varepsilon_{L}^\bullet$ preserves joins of separated sets. Consider a separated set $S$ in $L$. The set of connected elements below $\bigvee S$ is the union of the sets $C_s$ of connected elements below each $s\in S$, by definition of connectivity. Since the set $S$ is separated, for any two distinct $s\neq s'$, the sets $C_s$ and $C_{s'}$ are disjoint. Therefore, by Lemma~\ref{lemE}, the union $J=\bigcup\{C_s\mid s\in S\}$ is a subchainmail of $\mathcal{K}(L)$ and the join of $\{\varepsilon_{L}^\bullet(s)\mid s\in S\}$ is given by $J^\ast$. But at the same time, $J$ is the subchainmail of connected elements below $\bigvee S$, so $J^\ast = \varepsilon_{L}^\bullet(\bigvee S)$. Thus, $\varepsilon_{L}^\bullet$ preserves joins of separated sets. We have thus proved that each $\varepsilon_L$ is a connectivity homomorphism. By Lemma~\ref{lem:unique_separation} we also have that $\varepsilon_L$ is injective, and therefore a mono.

To establish naturality of $\varepsilon_L$, it suffices to show that for any connectivity homomorphism $F\colon L_1\to L_2$ we have $\varepsilon_{L_2}^\bullet F^\bullet=\mathcal{D}(\mathcal{K}(F))\varepsilon_{L_1}^\bullet$. For an element $y\in L_2$, we have: \[F^\bullet(y)=\bigvee\{x\in L_1\mid F(x)\le y\}.\]
Now, $\varepsilon_{L_2}^\bullet F^\bullet(y)$ is the totally disconnected set corresponding to the subchainmail $C_{F^\bullet(y)}$ of $\mathcal{K}(L_1)$ of connected elements $c$ in $L_1$ below $F^\bullet(y)$. Since $FF^\bullet (y)\le y$, the set $C_{F^\bullet(y)}$ is the same as the set of connected elements in $L_1$, which by $F$ are mapped below $y$. Since $F$ preserves connected elements, the set $C_{F^\bullet(y)}$ is the same as the set of connected elements in $L_1$ that are mapped to the subchainmail of $\mathcal{K}(L_2)$ consisting of connected elements below $y$. The totally disconnected set corresponding to this subchainmail of $\mathcal{K}(L_1)$ is exactly $\mathcal{D}(\mathcal{K}(F))\varepsilon_{L_1}^\bullet(y)$. 

It remains to prove that $\eta$ and $\varepsilon$ satisfy the triangle identities of adjunction, which is a fairly straightforward task: 
\begin{itemize}
    \item $\mathcal{K}(\varepsilon_L)$ will map a singleton $\{x\}$, where $x$ is a connected element of $L$, to $\bigvee\{x\}=x$, which is the inverse map of the isomorphism $\eta_{\mathcal{K}(L)}$.

    \item $\varepsilon_{\mathcal{D}(\Gamma)}$ will map a totally disconnected set $D$ of connected elements in $\mathcal{D}(\Gamma)$ to their join in $\mathcal{D}(\Gamma)$. Elements of such $D$ are singletons. Moreover, $\bigcup D$ is totally disconnected in $\Gamma$. Then, by Lemma~\ref{lemA}, $\bigcup D$ is the join of $D$ in $\mathcal{D}(\Gamma)$. So $\varepsilon_{\mathcal{D}(\Gamma)}(D)=\bigcup D$. This is indeed the inverse of $\mathcal{D}(\eta_\Gamma)$.\qedhere
\end{itemize}
\end{proof}

If a complete lattice $L$ has $\varepsilon_L$ an isomorphism, then it it is easy to see that it is a locally connected lattice. Conversely, if $L$ is locally connected, then $\mathcal{S}(L)=\mathcal{D}\mathcal{K}(L)$ and by Theorem~\ref{thmD}, $\varepsilon_L$ is an isomorphism.

\begin{corollary}
The equivalence of categories arising from the adjunction described in Theorem~\ref{thmI} is given by the full subcategory of $\mathbf{Con}$ consisting of locally connected lattices and the category $\mathbf{Chm}$ of chainmails.

\end{corollary}

\begin{remark}
    Using Lemma~\ref{lem:unique_separation}, we can prove the adjunction still holds if we weaken our definition of a connectivity homomorphism. The weaker definition would only require join-preservation and the preservation of connected elements.
\end{remark}

\begin{remark}
    There is a similar adjunction with the category of Serra lattices.
\end{remark}

\section{Conclusions and outlook}
In a similar vein to how frames and locales provide a point-free perspective on open sets, we have introduced chainmails to offer a point-free perspective on connected sets. Conceptually, a chainmail of some space should give us access to all regions that are a collection of disjoint connected components. We defined a locally connected complete lattice, conceptualized as a point-free counterpart to the subset of the power set of a locally connected space that contains these disjoint collections of connected components. Our work demonstrated an equivalence of categories between chainmails and such locally connected lattices. 

Connectedness is arguably a more `physical' and universally encountered concept than `openness', with natural examples found in graphs and other discrete structures. This insight paves the way for further exploration into point-free discrete spaces. Additionally, it suggests a novel approach to formalizing the continuum limit commonly employed in physics, potentially realized as a colimit in the category of chainmails.

Future research could further explore the applications of chainmails in modeling discrete and continuous spaces, explore their implications in mathematical physics, and investigate their relevance to other areas of mathematics. We note that there are known applications of point-free connectivity in image processing, as has been studied in mathematical morphology. By developing a point-free framework for connectedness, we open new avenues for understanding the interplay between discrete and continuous spaces. While the question of whether space is discrete or continuous on a more fundamental level is an open question in physics, we can now see that an independent question arises of whether space fundamentally has points (i.e. comes from a spatial frame), or only appears as such on large enough scales. Further understanding of the interactions of chainmails and frames will be needed, as will be explored in an upcoming paper.

\section*{Acknowledgments}
The authors would like to thank Cerene Rathilal and Nicholas Sander for various chainmail related discussions.

This work made use of resources provided by subMIT at MIT Physics and the University of Stellenbosch's \href{http://www.sun.ac.za/hpc}{HPC1 (Rhasatsha)}.

The first author would also like to thank MIT for financial support.

\bibliographystyle{unsrt} 
\bibliography{refs} 

\end{document}